\documentclass[english]{article}
\usepackage{babel}
\usepackage{amssymb}
\usepackage{amsmath}
\usepackage{amsthm}
\usepackage{pstricks}
\usepackage{graphicx}
\usepackage{pst-node}

\usepackage{times} 

\begin{document}
\def\R{\mathbb R}
\def\N{\mathbb N}
\def\Z{\mathbb Z}
\def\C{\mathbb C}
\def\Co{\mathcal{C}}
\def\F{\mathcal{F}}
\def\S{\mathcal{S}}
\def\1{\mathbf{1}}
\newtheorem{th-def}{Theorem-Definition}[section]
\newtheorem{theo}{Theorem}[section]
\newtheorem{lemm}[theo]{Lemma}
\newtheorem{prop}[theo]{Proposition}
\newtheorem{defi}[theo]{Definition}
\newtheorem{cor}[theo]{Corollary}
\newtheorem{exam}[theo]{Example}
\newtheorem{Rem}[theo]{Remark}
\newtheorem{Proof of}[theo]{Proof of main Theorem}
\def\dem{\noindent \textbf{Proof: }}
\def\rem{\indent \textsc{Remark. }}
\def\rems{\indent \textsc{Remarks. }}
\def\fin{  $\square$}
\def\ep{\varepsilon}
\def\uep{u^{\varepsilon}}
\def\Wep{W^{\varepsilon}}
\def\qep{a^{\varepsilon}}
\def\Sep{S^{\varepsilon}}
\def\Qep{Q^{\varepsilon}}
\def\Zep{Z^{\varepsilon}}
\def\vep{v^{\varepsilon}}
\def\wep{w^{\varepsilon}}
\def\fep{f^{\varepsilon}}
\def\gep{g^{\varepsilon}}
\def\hep{H^{\varepsilon}}
\def\aep{\alpha_{\varepsilon}}
\def\ph{\phi^{\ep}}
\def\X{\mathbb{X}}
\def\Y{\mathbb{Y}}
\def\H{\mathcal{H}}
\newcommand{\Node}[1]{\makebox[3mm]{#1}}
\def\cvput#1[#2]{\pnode(#1,1){#1} \pscircle*(#1,1){.1} \rput(#1,.5){$#2$}}
\def\vput#1{\cvput#1[#1]}

\newcommand{\edgebr}[1]{\makebox[3mm,linecolor=red]{#1}}

\numberwithin{equation}{section}
\title{Characterizations of some free random variables by properties of conditional moments of third degree polynomials}

\author{Wiktor Ejsmont}
\date{Mathematical Institute University of Wroclaw \\
pl. Grunwaldzki 2/4, 50-384 Wroc³aw, Poland\\
E-mail: wiktor.ejsmont@math.uni.wroc.pl}
\maketitle

{\small
\textbf{Abstract}. We 
 investigate  Laha-Lukacs properties  of noncommutative random variables (processes). We prove that  some families of free Meixner distributions can be characterized by the conditional moments of polynomial functions of degree 3. We also show that this fact has  consequences  in describing  some free L\'evy processes. The proof relies on a combinatorial identity. 
 At the end of this paper we show that this result can be extended to a $q$-Gausian variable. 
} \\
\\
\newpage
\section{Introduction}
The original motivation for this paper comes from a desire to understand the  results about the conditional expectation which were shown in  \cite{BoBr}, \cite{Br}, \cite{Ejs} and   \cite{SzWes}. 
  They   proved, that the first conditional linear  moment and 
 conditional quadratic variances  characterize free Meixner laws (Bo\.zejko and Bryc \cite{BoBr}, Ejsmont \cite{Ejs}). 
Laha-Lukacs type characterizations of random variables in free probability are also studied by Szpojankowski,  Weso\l owski \cite{SzWes}.
 They give a characterization of noncommutative
free-Poisson and free-Binomial variables by properties of the first two conditional moments,
which mimics Lukacs type assumptions known from classical probability. 
In this paper we show that free Meixner variables can be characterized by the third degree polynomial.
In particular,  we apply this result to describe a characterization of free L\'evy processes.

 In the last part of the paper we also show that these properties are also true for $q$-Gaussian variables. 
It is worthwhile to mention the work of Bryc \cite{Br}, where the Laha-Lukacs property for $q$-Gaussian processes was shown. Bryc proved that $q$-Gaussian processes have linear regressions and quadratic conditional variances. 

The paper is organized as follows. In section 2 we review basic free probability and free Meixner laws. We also establish a combinatorial identity used in the proof of the main theorem.  In section 3 we  proof our main theorem about the characterization of free Meixner distribution  by the conditional moments of polynomial functions of degree 3. In particular, we 
apply this result to describe a characterization of free L\'evy processes (and some property of this  processes).  Finally, in Section 4 we  compile some basic facts about a $q$-Gausian variable and we  show that the main result from Section 3 can be extended to a $q$-Gausian variable. 

\section{Free Meixner laws, free cumulants, conditional
expectation}
Classical Meixner distributions first appeared in the theory of orthogonal polynomials in the paper of Meixner \cite{Meix}. In free probability the  Meixner systems of polynomials were introduced by Anshelevich \cite{An1}, Bo\.zejko, Leinert, Speicher  \cite{BoLR} and Saitoh and Yoshida \cite{SY}. They showed  that the free Meixner system can be classified into six types of laws: the Wigner semicircle, the free Poisson, the free Pascal (free negative binomial), the free Gamma, a law that we will call pure free Meixner and the free binomial law. 

We assume that our probability space is a  von Neumann algebra $\mathcal{A}$ with a normal faithful tracial state $\tau:\mathcal{A} \to \mathbb{C}$ i.e., $\tau(\cdot)$ is linear, continuous in weak* topology, $\tau(\X \Y)=\tau(\Y \X)$,  $\tau(\mathbb{I})=1$, $\tau(\X\X^*)\geq 0$ and $\tau(\X \X^{*}) = 0$ implies $\X = 0$ for all $\X,\Y \in \mathcal{A}$.
 A (noncommutative) random variable $\mathbb{X}$ is a self-adjoint (i.e. $ \mathbb{X}=\mathbb{X}^*$) element of $\mathcal{A}$.  
We are interested in the two-parameter family of compactly supported probability measures (so that their moments does not grow faster than exponentially) $\{\mu_{a,b}: a \in \mathbb{R},b\geq -1 \} $ with the Cauchy-Stieltjes transform given by the formula

\begin{align}
G_\mu(z)=\int_{\mathbf{R}}\frac{1}{z-y}\mu_{a,b}(dy)=\frac{(1 + 2b)z + a -\sqrt{(z - a)^2 - 4(1 + b)}}{2(bz^2 + az + 1)}, \label{eq:transformataMixner}
\end{align}
\\
where the branch of the analytic square root should be determined by the condition
that $\Im(z)>0\Rightarrow \Im(G_\mu(z))\leqslant 0$ (see \cite{SY}). Cauchy-Stieltjes transform of $\mu$ is a function $G_\mu$ defined on the upper half
plane $\mathbb{C}^+=\{s+ti|s,t\in \mathbf{R}, t>0\}$ and takes values in the lower half plane $\mathbb{C}^-=\{s+ti|s,t\in \mathbf{R}, t\leq 0\}$. 
\\ Equation (\ref{eq:transformataMixner}) describes the distribution  with the mean equal to zero and the variance equal to one (see \cite{SY}).   The moment generating function,   which
 corresponds to the equation (\ref{eq:transformataMixner}),  has the form 

\begin{align}
M(z)=\frac{1}{z}G_\mu(\frac{1}{z})=\frac{1 + 2b + az -\sqrt{(1- za)^2 - 4z^2(1 + b)}}{2(z^2 + az + b)},  \label{eq:generujacamomenty}
\end{align}
for $|z|$ small enough. 
\\
\\
\noindent 
Let $\mathbb{C} \langle \mathbb{X}_{1},\dots ,\mathbb{X}_{n} \rangle$ denote the non-commutative ring of polynomials in variables $\mathbb{X}_{1},\dots ,\mathbb{X}_{n}$.
The free  cumulants are the $k$-linear maps $R_{k} : \mathbb{C} \langle \mathbb{X}_{1},\dots ,\mathbb{X}_{k} \rangle  \to\mathbb{C}$ defined  by the recursive  formula (connecting them with mixed moments)
\begin{align}
\tau(\mathbb{X}_{1}\mathbb{X}_{2}\dots \mathbb{X}_{n}) = \sum_{\nu \in NC(n)}R_{\nu}(\mathbb{X}_{1},\mathbb{X}_{2},\dots ,\mathbb{X}_{n}),\label{eq:DefinicjaKumulant}
\end{align}
where 
\begin{align}
R_{\nu}(\mathbb{X}_{1},\mathbb{X}_{2},\dots ,\mathbb{X}_{n}):=\Pi_{B \in \nu}R_{|B|}(\mathbb{X}_{i}:i \in B) 
\end{align}
and $NC(n)$ is the set of all non-crossing partitions of $\{1, 2,\dots, n \}$ (see \cite{NS,R}). Sometimes we will write $R_{k}(\mathbb{X})=R_{k}( \mathbb{X},\dots ,\mathbb{X} )$. 

\noindent The $\mathcal{R}$-transform of a random variable $\mathbb{X}$ is $\mathcal{R}_{\mathbb{X}}(z)=\sum_{i=0}^{\infty}R_{i+1}(\mathbb{X})z^i$, where  $R_{i}(\mathbb{X})$ is a sequences defined by (\ref{eq:DefinicjaKumulant})  (see \cite{BF} for more details).  For reader's convenience we recall that the $\mathcal{R}$-transform  corresponding to $M(z)$ which is equal to
   \begin{align}
\mathcal{R}_{\mu}(z)= \frac{2z}{1- za +\sqrt{(1- za)^2 - 4z^2b}}, \label{eq:RTransormataGlowna} 
\end{align}
where the analytic square root is chosen so that $\lim_{z \to 0}\mathcal{R}_{\mu}(z)=0$ (see \cite{SY}). If $\mathbb{X}$ has the distribution $\mu_{a,b}$, then sometimes we will write $\mathcal{R}_{\mathbb{X}}$ for the $\mathcal{R}$-transform of $\X$ . For particular values of $a$ and $b$ the law of $\X$ is:
\begin{itemize}
\item the Wigner's semicircle law if $a = b = 0$;
\item  the free Poisson  law if $b = 0$ and $a \neq 0$;
\item  the free Pascal (negative binomial) type law if $b > 0$ and $a^2 > 4b$;
\item  the free Gamma  law if $b > 0$ and $a^2 = 4b$;
\item the pure free Meixner  law if $b > 0$ and $a^2 < 4b$;
\item the free binomial  law  $-1\leq b < 0$.
\end{itemize}
\begin{defi}
\noindent Random variables $ \mathbb{X}_{1},\dots ,\mathbb{X}_{n} $  are freely independent (free) if, for every $n \geq 1$ and every non-constant choice of $\mathbb{Y}_{i} \in \{ \mathbb{X}_{1},\dots ,\mathbb{X}_{n}  \}$, where $i \in \{1,\dots,k\}$ (for each  $k=1,2,3\dots$) we get $R_{k}( \mathbb{Y}_{1},\dots ,\mathbb{Y}_{k} )=0$. 
\end{defi}

\noindent The $\mathcal{R}$-transform linearizes the free convolution, i.e. if $\mu$ and $\nu$ are (compactly supported) probability measures on $\mathbf{R}$, then we have 
\begin{align}
\mathcal{R}_{\mu \boxplus \nu}=\mathcal{R}_{\mu} + \mathcal{R}_{\nu},
\end{align}
where $\boxplus$ denotes the free convolution (the free convolution $\boxplus$  of measures ƒ$\mu,\nu$ is the law of $\X+\Y$ where
$\X,\Y$ are free and have laws ƒ$\mu,\nu$ respectively). For more details about free convolutions and free probability theory, the reader can consult \cite{NS,ViculecuDykemaNica}.
\\ \\
If $\mathcal{B} \subset \mathcal{A}$ is a von Neumann subalgebra and $\mathcal{A}$ has a trace $\tau$, then there exists a unique conditional
expectation from $\mathcal{A}$ to $\mathcal{B}$ with respect to $\tau$, which we denote by $\tau(\cdot|\mathcal{B})$. This map
is a weakly continuous, completely positive, identity preserving, contraction and it is characterized by the property that, for any $\mathbb{X} \in \mathcal{A}$, $\tau(\X\Y) = \tau(\tau(\X|\mathcal{B})\Y )$ for any $\mathbb{Y} \in \mathcal{B}$ (see \cite{Bian2,T}). For fixed $\X \in \mathcal{A}$ by $\tau(\cdot|\X)$ we denote the conditional expectation corresponding to the von Neumann algebra $\mathcal{B}$ generated by $\X$ and $\mathbb{I}$. 
The following lemma has been proven in \cite{BoBr}. 
\begin{lemm}
Let $\mathbb{W}$ be a (self-adjoint) element of the von Neumann algebra $\mathcal{A}$, generated by a self-adjoint $\mathbb{V} \in \mathcal{A}$. If, for all $n\geq  1$ we have $\tau(\mathbb{U}\mathbb{V}^n) =\tau(\mathbb{W}\mathbb{V}^n)$, then
\begin{align}
\tau(\mathbb{U}|\mathbb{V}) = \mathbb{W}.
\end{align} 
\label{lem:1}
 \end{lemm}
\noindent 
We introduce the notation  
\begin{itemize}
\item $NC(n)$ is the set of all non-crossing partitions of $\{1, 2,\dots , n \}$,
\item $NC^k(m)$ is the set of all non-crossing partitions of $\{1, 2,\dots , m\}$ (where $m\geq k\geq 1$) which have first $k$ elements in the same block. For example for  $k=3$ and $m=5$, see Figure \ref{fig:FiguraExemple1}.
\end{itemize}
\begin{figure}[h]
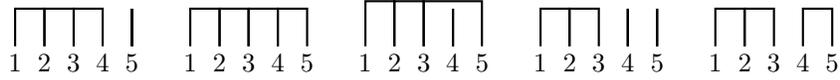

\begin{center}
\vspace{1cm}
\psset{nodesep=3pt}
\psset{angle=90}
\rnode{A}{\Node{$1$}} \rnode{B}{\Node{$2$}} \rnode{C}{\Node{$3$}}
\rnode{D}{\Node{$4$}} \rnode{E}{\Node{$5$}}
\ncbar[armA=.5]{A}{B}
\ncbar[armA=.5]{B}{C}
\ncbar[armA=.5]{C}{D}
\ncbar[armA=.5]{E}{E}
\hspace{3mm}
\rnode{A}{\Node{$1$}} \rnode{B}{\Node{$2$}} \rnode{C}{\Node{$3$}}
\rnode{D}{\Node{$4$}} \rnode{E}{\Node{$5$}}
\ncbar[armA=.5]{A}{B}
\ncbar[armA=.5]{B}{C}
\ncbar[armA=.5]{C}{D}
\ncbar[armA=.5]{D}{E}
\hspace{3mm}
\rnode{A}{\Node{$1$}} \rnode{B}{\Node{$2$}} \rnode{C}{\Node{$3$}}
\rnode{D}{\Node{$4$}} \rnode{E}{\Node{$5$}}
\ncbar[armA=.6]{A}{B}
\ncbar[armA=.6]{B}{C}
\ncbar[armA=.5]{D}{D}
\ncbar[armA=.6]{C}{E}
\hspace{3mm}
\rnode{A}{\Node{$1$}} \rnode{B}{\Node{$2$}} \rnode{C}{\Node{$3$}}
\rnode{D}{\Node{$4$}} \rnode{E}{\Node{$5$}}
\ncbar[armA=.5]{A}{B}
\ncbar[armA=.5]{B}{C}
\ncbar[armA=.5]{D}{D}
\ncbar[armA=.5]{E}{E}
\hspace{3mm}
\rnode{A}{\Node{$1$}} \rnode{B}{\Node{$2$}} \rnode{C}{\Node{$3$}}
\rnode{D}{\Node{$4$}} \rnode{E}{\Node{$5$}}
\ncbar[armA=.5]{A}{B}
\ncbar[armA=.5]{B}{C}
\ncbar[armA=.5]{D}{E}

\caption{Non-crossing partitions of $\{1, 2,3 , 4,5\}$ with the first $3$ elements in the same block.}
\label{fig:FiguraExemple1}
\end{center}
\end{figure}

The following lemma is a generalization of the Lemma 2.4 in \cite{Ejs} (the proof is also similar)

\begin{lemm}
Suppose that $\mathbb{Z}$ is a element of $ \mathcal{A}$, $m_{i}=\tau(\mathbb{Z}^{i})$ and $n,k\geqslant 1$.  Then
\begin{align}
\sum_{\nu \in NC^k(n+k)}R_{\nu}(\mathbb{Z})=\sum_{i=0}^{n-1}m_{i}\sum_{\nu \in NC^{k+1}(k+n-i)}R_{\nu}(\mathbb{Z})+R_k(\Z)m_{n}.
\end{align} 
\label{lem:2}
 \end{lemm}
\begin{Rem}
Note that in Lemma \ref{lem:2} we could only assume that $\mathbb{Z}$ is an element in a complex unital algebra $\mathcal{A}$  endowed with a linear function $\tau:\mathcal{A} \to \mathbb{C}$ satisfying $\tau(\mathbb{I})=1$.
\end{Rem}

\noindent  \textit{Proof of Lemma \ref{lem:2}.}
First, we consider partitions 
$\pi \in NC^k(n + k)$ with $\pi= \{V_1, \dots,V_s\}$ where $V_1=\{1,\dots,k\}$. The class of all, such  $\pi$ we will denote $ NC^k_0(n + k)$. It is clear, that the  sum of all  non-crossing partitions of this form corresponds to the term $R_k(\Z)m_{n}$.

On the other hand, for   $\nu \in NC^k(n + k)\backslash NC^k_0(n + k)$  
denote $s(\nu)=min\{j:j>k,j\in B_1\}$ where $ B_1$ is the block of $\nu $ which contains $1,\dots,k$.  
This decomposes $NC^k(n+k)$ into the $n$ classes $NC^k_{j}(n+k)= \{\nu \in NC^k(n+k) : s(\nu) = j\}, \textrm{ } j = k+1, \dots , n+k$. The set $ NC^k_{j}(n+k)$ can be identified with the product $NC(j-k-1)\times NC^{k+1}(n+2k-j+1)$ with convention that $NC(0)=\{\varnothing\}$. 
Indeed, the blocks of  $\nu \in NC^k_j(n+k)$, which partitions are the elements of $\{k+1, k+2, k+3, \dots ,j-1 \}$,  can be  identified with an appropriate partitions
in $NC(j-1-k)$, and (under the additional constraint that the first $k+1$ elements $1,\dots,k,j$ are in the same block)  the remaining blocks, which are partitions of the  set $\{1,\dots,k,j, j+1, . . . , n+k\}$,  can be uniquely identified with a partitions in $ NC^{k+1}(n+2k-j+1)$. The above situation is illustrated in  Figure \ref{fig:FiguraExemple2}.
\vspace{1 cm}
\begin{figure}[h]
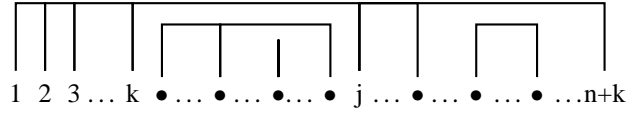

\begin{center}
\vspace{1cm}
\psset{nodesep=5pt}
\psset{angle=90}
\hspace{20mm}
\rnode{A}{\Node{1}} \rnode{A1}{\Node{2}}  \rnode{A2}{\Node{3}} \rnode{B}{\Node{\ldots}} \rnode{C1}{\Node{k}}
\rnode{D1}{\Node{$\bullet$}} \rnode{D5}{\Node{\ldots}} \rnode{D2}{\Node{$\bullet$}} \rnode{D5}{\Node{\ldots}} \rnode{D3}{\Node{$\bullet$}}\rnode{D5}{\Node{\ldots}} \rnode{D4}{\Node{$\bullet$}} \rnode{E}{\Node{j}} \rnode{E1}{\Node{\ldots}}  \rnode{E6}{\Node{$\bullet$}}  \rnode{E5}{\Node{\ldots}}  \rnode{E2}{\Node{$\bullet$}} \rnode{E3}{\ldots}\rnode{E7}{\Node{$\bullet$}}  \rnode{E3}{\ldots}  \rnode{F}{\Node{\textrm{ n+k }}} 
\ncbar[armA=.9]{A}{A1}
\ncbar[armA=.9]{A1}{A2}
\ncbar[armA=.9]{A2}{C1}
\ncbar[armA=.9]{C1}{E}
\ncbar[armA=.9]{E}{F}
\ncbar[armA=.9]{E}{E6}
\ncbar[armA=.7]{D1}{D2}
\ncbar[armA=.5]{D3}{D3}
\ncbar[armA=.7]{D2}{D4}
\ncbar[armA=.7]{E2}{E7}
\caption{The main structure of non-crossing partitions of $\{1, 2,3 ,\dots,n+k\}$ with the first $k$ elements in the same block.}
\label{fig:FiguraExemple2}
\end{center}
\end{figure}

\noindent This gives the formula 

\begin{align}
\sum_{\nu \in NC^k(n+k)}R_{\nu}(\mathbb{Z}) =\sum_{j=k+1}^{n+k}\sum_{\nu \in NC(j-k-1)}R_{\nu}(\mathbb{Z})\sum_{\nu \in NC^{k+1}(n+2k-j+1)}R_{\nu}(\mathbb{Z}) +R_k(\Z)m_{n} .\nonumber \end{align} 
Now we rewrite the last sum based on the value of $i=j-k-1$ where $i\in \{0,\dots,n-1 \}$. Thus, we have
\begin{align}
\sum_{\nu \in NC^k(n+k)}R_{\nu}(\mathbb{Z})&=\sum_{i=0}^{n-1}\sum_{\nu \in NC(i)}R_{\nu}(\mathbb{Z})\sum_{\nu \in NC^{k+1}(n+k-i)}R_{\nu}(\mathbb{Z}) +R_k(\Z)m_{n} \nonumber \\ &
=\sum_{i=0}^{n-1}m_{i}\sum_{\nu \in NC^{k+1}(k+n-i)}R_{\nu}(\mathbb{Z})+R_k(\Z)m_{n},
\end{align} 
which proves the lemma. 
\begin{flushright} $\square$ \end{flushright}

Let $\mathbb{Z}$ be the self-adjoint element of the von Neumann algebra $\mathcal{A}$  from the above lemma. We define $c^k_n=c^k_n(\mathbb{Z})=\sum_{\nu \in NC^k(n+k)}R_{\nu}(\mathbb{Z})$ and 
the following functions (power series):
\begin{align} C^k(z)=\sum_{n=0}^\infty c^k_nz^{k+n} \textrm{ where } k\geq 1 \end{align}
for sufficiently small $|z|<\epsilon$ and $z \in \mathbb{C}$. This series is convergent because we consider compactly supported probability measures, so moments and cumulants do not grow faster than exponentially (see \cite{BF}). This implies that $c^k_n$ also does not grow faster than exponentially. 
\begin{lemm}
Let $\mathbb{Z}$ be a (self-adjoint) element of the von Neumann algebra $\mathcal{A}$   then 
\begin{align}
C^{(k)}(z)=M(z)C^{(k+1)}(z)+R_k(\Z)z^kM(z)
\end{align} 
where $k\geqslant 1$.
\label{lem:4} 
 \end{lemm}
\begin{proof}
It is clear from Lemma \ref{lem:2} that we have
\begin{align} C^{(k)}(z)&=\sum_{n=0}^\infty c^k_n(\mathbb{Z})z^{k+n}=c^k_0(\mathbb{Z})z^{k}+\sum_{n=1}^\infty c^k_n(\mathbb{Z})z^{k+n}\nonumber \\& = c^k_0(\mathbb{Z})z^{k}+\sum_{n=1}^\infty [\sum_{i=0}^{n-1}m_{i}c^{k+1}_{n-i-1}(\mathbb{Z}) +R_k(\Z)m_n]z^{k+n}
 \nonumber \\&
= c^k_0(\mathbb{Z})z^{k}+\sum_{n=1}^\infty \sum_{i=0}^{n-1}m_{i}z^{i}c^{k+1}_{n-i-1}(\mathbb{Z})z^{k+n-i} +R_k(\Z)z^{k}\sum_{n=1}^\infty m_nz^{n}\nonumber \\& =M(z)C^{(k+1)}(z)+R_k(\Z)z^kM(z),
 \end{align}
which proves the lemma. 
\end{proof}
\begin{exam}  \noindent For $k=1$, we get:
\begin{align}
 C^{(1)}(z)=M(z)-1=M(z)C^{(2)}(z)+R_1(\Z)zM(z). \label{eq:exemK=1} 
 \end{align}\\
Particularly, we have the coefficients of the power series $1/M(z)$ (Maclaurin series): 
\begin{align}
 \frac{1}{M(z)}=1-C^{(2)}(z)-R_1(\Z)z \label{eq:exemK=12} 
 \end{align}
for sufficiently small $|z|$. \\
 Similarly, by putting $k=2$, we obtain:
\begin{align} 
C^{(2)}(z)=M(z)C^{(3)}(z)+R_2(\Z)z^2M(z). \label{eq:exemK=2} 
 \end{align}\end{exam}
\noindent Now we present  Lemma 4.1 of \cite{BoBr}, which will be used in the proof of the main theorem in order to calculate the moment generating function of free convolution.
\begin{lemm}[]
Suppose that $\mathbb{X}$,$\mathbb{Y}$  are free, self-adjoint and  $\mathbb{X}/\sqrt{\alpha}$,$\mathbb{Y}/\sqrt{\beta}$ have  the free Meixner  laws $\mu_{a/\sqrt{\alpha},b/\alpha}$ and  $\mu_{a/\sqrt{\beta},b/\beta}$ respectively, where $\alpha,\beta > 0 $, $\alpha+\beta =1 $ and $a \in \mathbb{R},b\geq -1$. Then  the moment generating function 
 $M(z)$ for $\X+\Y$ satisfies the following quadratic equation
\begin{align} (z^2 + az + b)M^2(z) - (1 + az + 2b)M(z) + 1 + b = 0. \label{eq:TransMomentSplot} \end{align}
\label{lem:3}
\end{lemm}
\section{Characterization of free Meixner laws }
The next lemma will be applied in the proof of the Theorem \ref{twr:1}.
\begin{lemm}

 If $\textrm{ }\X$ and $\Y$ are free independent and centered, then the condition $\beta R_{k}(\mathbb{X})=\alpha R_{k}(\mathbb{Y})$ for $\beta,\alpha>0$ and non-negative integers $k$ 
 is equivalent to 
\begin{align}
\tau(\mathbb{X}|(\mathbb{X}+\mathbb{Y})) =\frac{\alpha}{\alpha+\beta}(\mathbb{X}+\mathbb{Y}).\label{eq:WarunkowyPierwszyMoment}\end{align}
\begin{proof}
From the equation $\beta R_{k}(\mathbb{X})=\alpha R_{k}(\mathbb{Y})$ and from the  freeness of $\X$ and $\Y$ it stems that 
  \begin{align}
R_{k}(\mathbb{X})=\frac{\alpha}{\alpha+\beta}R_{k}(\mathbb{X}+ \mathbb{Y}). \label{eq:KumulantyZaleznoscX}
\end{align}
Analogously  we get
   \begin{align}
R_{k}(\mathbb{Y})=\frac{\beta}{\alpha+\beta}R_{k}(\mathbb{X}+ \mathbb{Y}). \label{eq:KumulantyZaleznoscY}
\end{align}
This gives 
\begin{align}
\tau(\mathbb{X}(\mathbb{X}+\mathbb{Y})^{n}) &= \sum_{\nu \in NC(n+1)}R_{\nu}(\mathbb{X},\mathbb{X}+\mathbb{Y},\dots ,\mathbb{X}+\mathbb{Y}) \nonumber \\&
=\frac{\alpha}{\alpha+\beta} \sum_{\nu \in NC(n+1)}R_{\nu}(\mathbb{X}+\mathbb{Y},\mathbb{X}+\mathbb{Y},\dots ,\mathbb{X}+\mathbb{Y}) \nonumber \\&
=\tau(\frac{\alpha}{\alpha+\beta}(\mathbb{X}+\mathbb{Y})(\mathbb{X}+\mathbb{Y})^{n})
\end{align}
which, by Lemma \ref{lem:1}, implies that $\tau(\mathbb{X}|(\mathbb{X}+\mathbb{Y}))=\frac{\alpha}{\alpha+\beta}(\mathbb{X}+\mathbb{Y})$.
Let's suppose that the assertion (\ref{eq:WarunkowyPierwszyMoment}) is true. Then we use first part of the proof of Theorem 3.2 from the article \cite{BoBr}. From this proof (the first part, by induction)  we can deduce that the condition (\ref{eq:WarunkowyPierwszyMoment}) implies $\beta R_{k}(\mathbb{X})=\alpha R_{k}(\mathbb{Y})$.
\end{proof} \label{rem:1}
\end{lemm}
The main result of this paper is the following characterization of free Meixner laws in the terms of the cubic polynomial condition for conditional moments. 
\begin{theo}[]
Suppose that $\mathbb{X}$,$\mathbb{Y}$  are free, self-adjoint, non-degenerate, centered $(\tau(\mathbb{X})=\tau(\mathbb{Y})=0)$ and $\tau(\mathbb{X}^2+\mathbb{Y}^2)=1$. Then $\mathbb{X}/\sqrt{\alpha}$ and  $\mathbb{Y}/\sqrt{\beta}$ have the free Meixner  laws $\mu_{a/\sqrt{\alpha},b/\alpha}$ and  $\mu_{a/\sqrt{\beta},b/\beta}$, respectively, where $a \in \mathbb{R},b\geq -1$ if and only if  
\begin{align}
\tau(\mathbb{X}|(\mathbb{X}+\mathbb{Y})) ={\alpha}(\mathbb{X}+\mathbb{Y}) \label{eq:WarunkowyPierwszyMomentMainTwr}\end{align} 
and
\begin{align} 
&\tau((\beta\mathbb{X}-\alpha\mathbb{Y})(\mathbb{X}+\mathbb{Y})(\beta\mathbb{X}-\alpha\mathbb{Y})|(\mathbb{X}+\mathbb{Y}))\nonumber \\ &=
 \frac{\alpha\beta}{(b+1)^2} (b^2(\mathbb{X}+\mathbb{Y})^{3}+2ba(\mathbb{X}+\mathbb{Y})^{2}+(b+a^2)(\mathbb{X}+\mathbb{Y})+a\mathbb{I})    \label{eq:warjancjawarunkowa2}
\end{align} 
for some  $\alpha,\beta > 0 $ and $\alpha+\beta = 1$.
Additionally, we assume that $b\geq max\{-\alpha, -\beta\}$ if $b<0$ (the free binomial case).
\label{twr:1}
 \end{theo}
  \begin{Rem} In commutative probability equation \eqref{eq:warjancjawarunkowa2} takes the form:
\begin{align}
\tau((\beta\mathbb{X}-\alpha\mathbb{Y})(\mathbb{X}+\mathbb{Y})(\beta\mathbb{X}-\alpha\mathbb{Y})|(\mathbb{X}+\mathbb{Y}))\nonumber =c(\mathbb{X}+\mathbb{Y})^{3}+d(\mathbb{X}+\mathbb{Y})^{2}+e(\mathbb{X}+\mathbb{Y})
\end{align}
for some $c,d,e \in \mathbb{R}$, which is equivalent to the assumption that the conditional variance is quadratic.  There are also a higher degree polynomial
regression studied in commutative probability, see e.g. \cite{Bar-Lev1,Bar-Lev2,FosamShanbhag1,Gordon}. They proved that some classical random variable can be characterized by the higher degree polynomial but in a different context as presented in this article. 
In free probability the result \eqref{eq:warjancjawarunkowa2} is in some way unexpected.
As an argument we can give the Wigner's semicircle law variables. Suppose that $\mathbb{X}$,$\mathbb{Y}$  are free, self-adjoint, non-degenerate, centered $(\tau(\mathbb{X})=\tau(\mathbb{Y})=0)$, $\tau(\mathbb{X}^2)=\tau(\mathbb{Y}^2)=1$ and have the same distribution.  
Then the following statements are equivalent:
\begin{itemize}
\item $\mathbb{X}$ and  $\mathbb{Y}$ have the Wigner's semicircle law,
\item $\tau((\mathbb{X}-\Y)^2|\mathbb{X}+\mathbb{Y}) =2\mathbb{I}$ -- which follows from the main Theorem of \cite{BoBr} and \cite{Ejs},
 
\item 
$ 
\tau((\mathbb{X}-\mathbb{Y})(\mathbb{X}+\mathbb{Y})(\mathbb{X}-\mathbb{Y})|(\mathbb{X}+\mathbb{Y}))=\mathbb{O}$ -- which follows from the Theorem \ref{twr:1} ($a=b=0$).

\end{itemize}
Thus we see that the last  equation is unexpected, because in the classical case from $\tau((\mathbb{X}-\Y)^2|\mathbb{X}+\mathbb{Y}) =2$ we can easily deduce $ 
\tau((\mathbb{X}-\mathbb{Y})(\mathbb{X}+\mathbb{Y})(\mathbb{X}-\mathbb{Y})|(\mathbb{X}+\mathbb{Y}))=2(\mathbb{X}+\mathbb{Y})$, and in fact in noncommutative probability, the conditional expectation $ \tau(\X \Y \Z|\Y)$ is difficult to compute (if we know $ \tau(\X\Z|\Y)$).
 \end{Rem}
\noindent  \textit{Proof of Theorem \ref{twr:1}.}
$\Rightarrow$: Suppose that $\mathbb{X}/\sqrt{\alpha}$ and  ,$\mathbb{Y}/\sqrt{\beta}$ have respectively the free Meixner  laws $\mu_{a/\sqrt{\alpha},b/\alpha}$ and  $\mu_{a/\sqrt{\beta},b/\beta}$.
The condition \eqref{eq:WarunkowyPierwszyMomentMainTwr} holds because we can use Theorem 3.1 from
the article \cite{Ejs}.
 Then, from Lemma \ref{lem:3} the moment generating functions satisfy equation (\ref{eq:TransMomentSplot}). If in (\ref{eq:TransMomentSplot})  we multiply by $(1 -C^{(2)}(z))$ both sides and use the fact (\ref{eq:exemK=1}) with  $R_1(\X+\Y)=0$, we get
\begin{align}
 M(z)(b+za+z^2) -(2b+1+za)+(b+1)(1 -C^{(2)}(z))=0  \label{eq:pomcniczy1}
\end{align}
where $C^{(2)}(z)$ is a function for $\X+\Y$. Expanding $M(z)$ in the series ($M(z)=1+\sum_{i=1}^{\infty}z^{i}m_{i}$), we get 
\begin{align}
 bm_{n+2}+am_{n+1}+m_{n}=(b+1)c^{(2)}_n. \label{eq:Pomiedzyc_2Am_n}
\end{align}
Now we apply (\ref{eq:exemK=2}) to the equation (\ref{eq:pomcniczy1}) (using the assumption $R_2(\X+\Y)=1$) and after simple computations, we see 
that \begin{align}
 M(z)(b+za+z^2) -(b+za)=(b+1)(M(z)C^{(3)}(z)+z^2M(z))
\end{align}
or equivalently:
\begin{align}
  b+za-z^2b -\frac{(b+za)}{M(z)}=(b+1)C^{(3)}(z)
\end{align}
for $|z|$ small enough. Then using  (\ref{eq:exemK=12}) we have 
\begin{align}
  -z^2b +(b+za)C^{(2)}(z)=(b+1)C^{(3)}(z).
\end{align}
Expanding the  above equation in series, we get

\begin{align}
 bc^{(2)}_{n+1} +ac^{(2)}_n=c^{(3)}_n(b+1),
\end{align}
and using (\ref{eq:Pomiedzyc_2Am_n}) we obtain
\begin{align}
 b^2m_{n+3}+2bam_{n+2}+(b+a^2)m_{n+1}+am_{n} =c^{(3)}_n(b+1)^2. \label{eq:zaleznoscszereg_a3}
\end{align}
From the assumption of the main Theorem and Lemma \ref{rem:1} we get
\begin{align} R_{k}(\beta\mathbb{X}-\alpha\mathbb{Y},\mathbb{X+Y},\mathbb{X+Y},\dots,\mathbb{X+Y})=\beta R_{k}(\X)-\alpha R_{k}(\Y)=0  \label{eq:zalozeniekumulanty}
\end{align}
and similarly for $k \geqslant 3$
\begin{align} 
&R_{k}(\beta\mathbb{X}-\alpha\mathbb{Y},\mathbb{X+Y},\beta\mathbb{X}-\alpha\mathbb{Y},\mathbb{X+Y},\dots,\mathbb{X+Y})=\beta^2 R_{k}(\X)+\alpha^2 R_{k}(\Y)\nonumber \\& =\beta \alpha R_{k}(\X+\Y). \label{eq:zalozeniekumulanty2}
\end{align}
Now we use the moment-cumulant formula (\ref{eq:DefinicjaKumulant})
 \begin{align}
&\tau((\beta\mathbb{X}-\alpha\mathbb{Y})(\mathbb{X+Y})(\beta\mathbb{X}-\alpha\mathbb{Y})(\mathbb{X+Y})^n)\nonumber \\ &=
\sum_{\nu \in NC(n+3)}R_{\nu}(\beta\mathbb{X}-\alpha\mathbb{Y},\mathbb{X+Y},\beta\mathbb{X}-\alpha\mathbb{Y},\underbrace{ \mathbb{X+Y},\mathbb{X+Y},\dots,\mathbb{X+Y}}_{n\textrm{-times}}) \nonumber \\& =
\sum_{\nu \in NC^3(n+3)}R_{\nu}(\beta\mathbb{X}-\alpha\mathbb{Y},\mathbb{X+Y},\beta\mathbb{X}-\alpha\mathbb{Y}, \mathbb{X+Y},\mathbb{X+Y},\dots,\mathbb{X+Y}) \nonumber\\& +\sum_{\nu \in NC(n+3) \diagdown NC^3(n+3)}R_{\nu}(\beta\mathbb{X}-\alpha\mathbb{Y},\mathbb{X+Y},\beta\mathbb{X}-\alpha\mathbb{Y}, \mathbb{X+Y},\mathbb{X+Y},\dots,\mathbb{X+Y}).\nonumber 
 \end{align}
Let us  look more closely at the second sum from the last equation. We have that either the first and the third elements are in different blocks, or they are in the same
block. In the first case,  the second sum (from the last equation) vanishes because we have \eqref{eq:zalozeniekumulanty}. On the other hand, if they are in the same block, the  sum disappears because then we have that $\tau(\X+\Y)=0$. So, by (\ref{eq:zalozeniekumulanty2}) we have
 \begin{align}
&\tau((\beta\mathbb{X}-\alpha\mathbb{Y})(\mathbb{X+Y})(\beta\mathbb{X}-\alpha\mathbb{Y})(\mathbb{X+Y})^n)\nonumber \\ &=
\alpha \beta\sum_{\nu \in NC^3(n+3)}R_{\nu}(\mathbb{X}+\mathbb{Y},\mathbb{X+Y},\mathbb{X}+\mathbb{Y},\underbrace{ \mathbb{X+Y},\mathbb{X+Y},\dots,\mathbb{X+Y}}_{n\textrm{-times}})\nonumber \\ &=\alpha \beta c^{(3)}_n .\label{eq:LaczenieKumulantOrazGlownegowyniku}
 \end{align}
Therefore the equation (\ref{eq:zaleznoscszereg_a3}) is equivalent to
\begin{align}
& \alpha\beta \tau(b^2(\mathbb{X}+\mathbb{Y})^{n+3}+2ba(\mathbb{X}+\mathbb{Y})^{n+2}+ (b+a^2)(\mathbb{X}+\mathbb{Y})^{n+1}+a(\mathbb{X}+\mathbb{Y})^{n} )\nonumber \\ & =\tau((\beta\mathbb{X}-\alpha\mathbb{Y})(\mathbb{X}+\mathbb{Y})(\beta\mathbb{X}-\alpha\mathbb{Y})(\mathbb{X}+\mathbb{Y})^{n})(b+1)^2. 
\end{align}
for all $n \geqslant 0$. Now we use Lemma \ref{lem:1} to get  (\ref{eq:warjancjawarunkowa2}).
%
 \\
$\Leftarrow$: Let's suppose now, that the equality (\ref{eq:WarunkowyPierwszyMomentMainTwr}) and (\ref{eq:warjancjawarunkowa2}) holds. Multiplying (\ref{eq:warjancjawarunkowa2}) by $(\mathbb{X}+\mathbb{Y})^n$ for $n \geqslant 0$ and applying $\tau(\cdot)$ we obtain (\ref{eq:zaleznoscszereg_a3}). Using the facts that $m_1 = 0$ and $m_2 = 1$, from (\ref{eq:zaleznoscszereg_a3}) we obtain

\begin{align}
 & b^2M(z)+2zbaM(z)+z^2(b+a^2)M(z)+z^3aM(z) \nonumber \\&
-b^2z^2-b^2-2zba-z^2(b+a^2)=C^{(3)}(z)(b+1)^2.
\end{align}
From (\ref{eq:exemK=2}) we get

\begin{align}
 & b^2M^2(z)+2zbaM^2(z)+z^2(b+a^2)M^2(z)+z^3aM^2(z) \nonumber \\ &
-(b^2z^2+b^2+2zba+z^2(b+a^2))M(z)+z^2M(z)(b+1)^2=C^{(2)}(z)(b+1)^2,
\end{align}
and from (\ref{eq:exemK=1}) we have
\begin{align}
& b^2M^3(z)+2zbaM^3(z)+z^2(b+a^2)M^3(z)+z^3aM^3(z) \nonumber \\&
-(b^2z^2+b^2+2zba+z^2(b+a^2))M^2(z)+z^2M^2(z)(b+1)^2=(b+1)^2 (M(z)-1),
\end{align}
or equivalently
\begin{align}
 &((b+za)M(z)+b+1)\times\nonumber \\&((b+za+z^2)M^2(z) - M(z)(2b+za+1)+b+1)=0.
\end{align}
Thus, we have found two solutions (if $a\neq 0$ and $b \neq 0$)
\begin{align}
 M(z)=-(b+1)/(b+za)  \end{align} or \begin{align}(b+za+z^2)M^2(z) - M(z)(2b+za+1)+b+1=0 \label{eq:GoodSolution}
\end{align}
but the first solution does not corresponds to probability measure (except for $b=-0.5$) because then $M(z)=-\frac{b+1}{b}+-\frac{b+1}{b}\sum_{i=1}^{\infty}(\frac{za}{b})^n$. If $b=-0.5$ then the solution corresponds to the Dirac measure at the point $-2a$. However, by the assumption that the variable is non-degenerate variable we reject this solution. 
Thus we have \eqref{eq:GoodSolution}  and Lemma \ref{lem:3} says that $\X$ and $\Y$ have the Meixner laws,
which completes the proof.\begin{flushright} $\square$ \end{flushright}

\noindent 
 A non-commutative stochastic process ($\X_{t}$) is a free L\'evy process, if it has free additive and stationary increments. For a more detailed discussion of free L\'evy processes we refer to \cite{BT}. Let us first recall 
some properties of free L\'evy processes which follow from 
\cite{BoBr}. If $(\X_t)$ is a free L\'evy process such that $\tau(\X_t) = 0$ and $\tau(\X^2_t)=t$ for all $t > 0$ then
  \begin{align} 
\tau(\mathbb{X}_{s}|\mathbb{X}_{u})=
 \frac{s}{u} \mathbb{X}_{u} \label{eq:WzorBozBryc}
\end{align}
for all $s<u$. This note allows to formulate the following proposition.

\begin{prop}[] Suppose that $(\mathbb{X}_{t\geq 0})$  is a free L\'evy process such as  $\tau(\X_t) = 0$ and $\tau(\X^2_t)=t$ for all $t > 0$. Then the increment $(\mathbb{X}_{t+s}-\mathbb{X}_{t})/\sqrt{s}$ $(t,s>0)$ has the free Meixner law $\mu_{a/\sqrt{s},b/s}$ (for some $b\geqslant 0$) if and only if for all $t < s$
  \begin{align} 
\tau(\mathbb{X}_{t}\mathbb{X}_{s}\mathbb{X}_{t}|\mathbb{X}_{s}) =
 \frac{(s-t)t}{s^2(b+s)^2} (b^2\mathbb{X}_{s}^{3}+2bas\mathbb{X}_{s}^{2}+(b+a^2)s^2\mathbb{X}_{s}+as^3\mathbb{I} )+\frac{t^2}{s^2}\mathbb{X}^3_{s}
\end{align} 
\label{twr:4}
 \end{prop}
\begin{Rem}
The existence of a free L\'evy process was demonstrate by Biane \cite{Bian2} who proved that from every infinitely divisible distribution we can construct a free L\'evy process. We assume that $b\geqslant 0$ in Proposition \ref{twr:4} because a free Meixner variable is infinitely divisible if and only if $b\geqslant 0$ (see \cite{AnMlotko,BoBr}).
\end{Rem}
\noindent  \textit{Proof of Proposition \ref{twr:4}}.
Let's rewrite Theorem \ref{twr:1} for the variables (non-degenerate) $\X$ and $\Y$ such that $\tau(\mathbb{X}^2)=\alpha$, $\tau(\mathbb{Y}^2)=\beta$ and $\tau(\mathbb{Y})=\tau(\mathbb{X})=0$. After a simple parameter normalization   ($\alpha$ by $\frac{\alpha}{\alpha+\beta}$, $\beta$ by $\frac{\beta}{\alpha+\beta}$, $a$ by $\frac{a}{\sqrt{\alpha+\beta}}$, $b$ by $\frac{b}{\alpha+\beta}$) we get that $\X/\sqrt{\alpha}=\frac{\X}{\sqrt{\alpha+\beta}}/\frac{\sqrt{\alpha}}{\sqrt{\alpha+\beta}}$ and $\Y/\sqrt{\beta}=\frac{\Y}{\sqrt{\alpha+\beta}}/\frac{\sqrt{\beta}}{\sqrt{\alpha+\beta}}$ have the free Meixner laws  $\mu_{a/\sqrt{\alpha},b/\alpha}$ and  $\mu_{a/\sqrt{\beta},b/\beta}$, respectively, if and only if (after a simple computation)
\begin{align} 
&\tau((\beta\mathbb{X}-\alpha\mathbb{Y})(\mathbb{X}+\mathbb{Y})(\beta\mathbb{X}-\alpha\mathbb{Y})|(\mathbb{X}+\mathbb{Y}))\nonumber \\ &=
 \frac{\alpha\beta}{(b+(\alpha+\beta))^2} (b^2(\mathbb{X}+\mathbb{Y})^{3}+2ba(\alpha+\beta)(\mathbb{X}+\mathbb{Y})^{2}\nonumber \\ &+(\alpha+\beta)^2(b+a^2)(\mathbb{X}+\mathbb{Y})+a(\alpha+\beta)^3\mathbb{I} ).    \label{eq:warjancjaProces}
\end{align}  
i.e. we apply  Theorem  \ref{twr:1} with $\X$ equal $\frac{\X}{\sqrt{\alpha+\beta}}$ and $\Y$  equal $\frac{\Y}{\sqrt{\alpha+\beta}}$ and the parameters  mentioned above in the brackets. 
Now we consider  two variables $\mathbb{X}_{t}/\sqrt{t}$ and $(\mathbb{X}_{s}-\mathbb{X}_{t})/\sqrt{s-t}$, which are free and centered. 
 Thus, the formula (\ref{eq:warjancjaProces}) tell us that  $\mathbb{X}_{t}/\sqrt{t}$ and $(\mathbb{X}_{s}-\mathbb{X}_{t})/\sqrt{s-t}$ ($\X=\X_t$, $\Y=\Y_t$, $\alpha=t$, $\beta=s-t$), have the free Meixner laws  $\mu_{a/\sqrt{t},b/t}$ and  $\mu_{a/\sqrt{s-t},b/(s-t)}$, respectively, if and only if
 \begin{align} 
\tau((t\mathbb{X}_{s}-s\mathbb{X}_{t})\mathbb{X}_{s}(t\mathbb{X}_{s}-s\mathbb{X}_{t})|\mathbb{X}_{s})&\stackrel{(\ref{eq:WzorBozBryc}) }{=}
t^2\mathbb{X}^3_{s}-t^2\mathbb{X}^3_{s} +s^2\tau(\mathbb{X}_{t}\mathbb{X}_{s}\mathbb{X}_{t}|\mathbb{X}_{s})-t^2\mathbb{X}^3_{s}\nonumber \\&=
 \frac{(s-t)t}{(b+s)^2} (b^2\mathbb{X}_{s}^{3}+2bas\mathbb{X}_{s}^{2}+(b+a^2)s^2\mathbb{X}_{s}+as^3\mathbb{I} ).
\end{align} 
Thus, the proposition holds.
\begin{flushright} $\square$ \end{flushright}
At the end of this section, we are coming to the following proposition.
\begin{prop}[]
Suppose that $(\mathbb{X}_{t\geq 0})$  is a free L\'evy process such that the increments $(\mathbb{X}_{t+s}-\mathbb{X}_{t})/\sqrt{s}$ $(t,s>0)$ have the free Meixner law $\mu_{a/\sqrt{s},b/s}$ (for some $b\geqslant 0$). Then 
\begin{align} 
\tau(\mathbb{X}^3_{t}|\mathbb{X}_{2t})&=\frac{1}{8(b+2t)^2} (b^2\mathbb{X}_{2t}^{3}+4bat\mathbb{X}_{2t}^{2}+4(b+a^2)t^2\mathbb{X}_{2t}+8at^3\mathbb{I} )\nonumber\\&+\frac{1}{8} \mathbb{X}^3_{2t}+\frac{1}{4(b+2t)}
 [4t^2\mathbb{X}_{2t}+2ta\mathbb{X}_{2t}^2+b\mathbb{X}_{2t}^{3}] .   
 \label{eq:RownaniePropozycja2}
\end{align} 

 \end{prop}

\begin{proof}Let $\mathbb{X}_{t}$ be as in the above proposition. 
First, we show that the third conditional central moment is equal to zero i.e. 
\begin{align} 
\tau((\mathbb{X}_{t}-\tau(\mathbb{X}_{t}|\mathbb{X}_{2t}))^3|\mathbb{X}_{2t})=0.
\end{align} 
for all $t>0$.
From the assumption we have $\tau((\mathbb{X}_{t}-\frac{t}{2t} \mathbb{X}_{2t}))^3|\mathbb{X}_{2t})=\tau((2t\mathbb{X}_{t}-t\mathbb{X}_{2t}))^3|\mathbb{X}_{2t})/(2t)^3$.  For this reason,  $\tau((2t\mathbb{X}_{t}-t\mathbb{X}_{2t}))^3\mathbb{X}^k_{2t})=0$ for all integers $k\geq 0$, by the relation $tR_k(\X_{2t})=2tR_k(\X_t)$. Indeed, if the first element ($2t\mathbb{X}_{t}-t\mathbb{X}_{2t}$) is  in the partition with  the element  from ''part'' $\mathbb{X}^k_{2t}$ only then  we have the cumulant 
\begin{align} &R_k(2t\X_t-t\X_{2t},\X_{2t},\dots,\X_{2t}) = {2t}R_k(\X_t,\X_{2t},\dots,\X_{2t})-tR_k(\X_{2t},\X_{2t},\dots,\X_{2t})
\nonumber \\&= {2t}R_k(\X_t,\X_{2t}-\X_t+\X_t,\dots,\X_{2t}-\X_t+\X_t)-tR_k(\X_{2t},\X_{2t},\dots,\X_{2t})
\nonumber \\&={2t}R_k(\X_t,\X_t,\dots,\X_t)-tR_k(\X_{2t},\X_{2t},\dots,\X_{2t})=0.
\end{align} 
Now, if the first element  is in the  partition with the second or  third element (but not simultaneously) then  cumulants are zero as well (by a similar argument presented above). Thus, the first three elements must be in the same block, so using the fact ${2t}\X_t-t\X_{2t}=t\X_t-t(\X_{2t}-\X_t)$ and $\X_{2t}=\X_{2t}-\X_t+\X_t$ ($\X_{2t}-\X_t$ and $\X_t$ are free) we obtain 
\begin{align} & R_k({2t}\X_t-t\X_{2t},{2t}\X_t-t\X_{2t},{2t}\X_t-t\X_{2t},\X_{2t},\dots,\X_{2t})
\nonumber \\&= t^3R_k(\X_t,\X_t,\dots,\X_t)-t^3R_k(\X_{2t}-\X_t,\X_{2t}-\X_t,\dots,\X_{2t}-\X_t)
\nonumber \\&=t^3R_k(\X_t)-t^3R_k(\X_{2t})+t^3R_k(\X_t)=t^2(2tR_k(\X_t)-tR_k(\X_{2t}))=0.
\end{align} 
From Lemma \ref{lem:1} we obtain $\tau((\mathbb{X}_{t}-\frac{1}{2} \mathbb{X}_{2t}))^3|\mathbb{X}_{2t})=0 $, or equivalently 
\begin{align} 0&=\tau((2\mathbb{X}_{t}- \mathbb{X}_{2t}))^3|\mathbb{X}_{2t})\nonumber \\&= \tau(8\mathbb{X}^3_{t}|\mathbb{X}_{2t})- 4\tau(\mathbb{X}^2_{t}|\mathbb{X}_{2t})\mathbb{X}_{2t} -4\mathbb{X}_{2t}\tau(\mathbb{X}^2_{t}|\mathbb{X}_{2t})-4\tau(\mathbb{X}_{t}\mathbb{X}_{2t}\mathbb{X}_{t}|\mathbb{X}_{2t})+2\mathbb{X}^3_{2t}.   
\label{eq:RownaniePropozycja2Pom}
\end{align}
To compute $\tau(\mathbb{X}_{t}\mathbb{X}_{2t}\mathbb{X}_{t}|\mathbb{X}_{2t})$  we use Proposition \ref{twr:4} and to compute the expression $\tau(\mathbb{X}^2_{t}|\mathbb{X}_{2t})$ we use  Proposition 3.2. from \cite{Ejs}. Here we don't cite  this proposition (Proposition 3.2. from the paper \cite{Ejs}), we  note only that if we know $\tau(\mathbb{X}_{t}\mathbb{X}_{2t}\mathbb{X}_{t}|\mathbb{X}_{2t})$  and $\tau(\mathbb{X}^2_{t}|\mathbb{X}_{2t})$ then from \eqref{eq:RownaniePropozycja2Pom} we can compute $\tau(\mathbb{X}_{t}^3|\mathbb{X}_{2t})$ (we skip simple calculations leading to the formula \eqref{eq:RownaniePropozycja2}). This completes the proof.
\end{proof}

\section{Some consequences for a $q$-Gaussian Random Variable }

In this section we consider a  mapping $\mathcal{H}\ni f \rightarrow G_{f}\in \mathcal{B}(\H)$ from a real
Hilbert space $\mathcal{H}$ into the algebra $\mathcal{B}(\H)$ of bounded operators acting on the space $\H$. We also use a parameter $q\in(-1,1)$.  We consider non-commutative random variables as the elements of the von Neumann algebra $\mathcal{A}$, generated by the bounded (i.e $-1<q<1$), self-adjoint operators $G_f$, with a state $\mathbb{E}:\mathcal{A} \rightarrow \mathbb{C}$. State $\mathbb{E}$ is a unital linear functional (which means that it preserves the identity), positive (which means  $\mathbb{E}(\X)\geqslant 0$  whenever $\X$ is a non-negative element of $\mathcal{A}$), faithful (which means  that if  $\mathbb{E}(\Y^*\Y)=0$ then $\Y = 0$), and not necessarily tracial. 
  In $(\mathcal{A},\mathbb{E})$ we refer to the self-adjoint elements of the algebra  $\mathcal{A}$ as random
variables. Similarly as in free probability any self-adjoint random variable $\X$ has a law: this is the unique compactly supported probability
measure $\mu$ on $\mathbb{R}$ which has the same moments as $\X$ i.e. $\tau(\X^n)=\int t^n d\mu(t)$, $n=1,2,3,\dots$. 
 
Denote by $P_n$ the lattice of all partitions of $\{1,\dots,n\}$. Fix a partition $\sigma \in P_n$, with blocks $\{B_1,\dots,B_k\}$. For a block $B$, denote by $a(B)$ its first element. Following \cite{Bian1}, we define the number of restricted crossings of a partition $\sigma$ as follows. For $B$ a block of $\sigma$ and $i \in B$, $i\neq a(B)$, denote $p(i) = max\{j \in B, j < i\}$. For two blocks $B,C \in \sigma$, a restricted crossing
is a quadruple ($p(i) < p(j) < i < j$) with $i \in B$, $j \in C$.The number of restricted crossings of $B,C$ is 
 \begin{align}
rc(B,C)= |\{i \in B, j \in C:p(i) < p(j) < i < j\}|+|\{j \in B, i \in C:p(j) < p(i) < j < i\}|,  \label{eq:RestrictedCrossings}
\end{align}
and the number of restricted crossings of $\sigma$  is $rc (\sigma ) =\Sigma_{ i<j} rc (B_i,B_j)$ (see also \cite{An2,An3}).
\begin{defi} For a seqence $A_{f_1},A_{f_2}\dots $ let $\textrm{ }\mathbb{C}\langle A_{f_1},A_{f_2},\dots,A_{f_n}\rangle$ denote the non-commutative ring of polynomials in variables $A_{f_1},A_{f_2},\dots,A_{f_n}$.
The $q$-deformed cumulants are the $k$-linear maps $R^{q}_{k} : \mathbb{C}\langle A_{f_1},A_{f_2},\dots,A_{f_n}\rangle  \to\mathbb{C}$ defined by the recursive formula 
 \begin{align}
\mathbb{E}(A_{f_1}\dots A_{f_n})= \sum_{\sigma \in P_n}q^{rc(\sigma)}\prod_{B \in \sigma } R^{q}_{|B|}(A_{f_i}:i \in B).     \label{eq:qCumulanty}
\end{align}
\end{defi}
\noindent To state Theorem \ref{twr:TwierdzenieDlaqasian} we shall need the following definition (see also \cite{An2,An3}). 
\begin{defi} Random variables $A_{f_1},A_{f_2},\dots,A_{f_n} $  are $q$-independent if, for every $n \geq 1$ and every non-constant choice of $\mathbb{Y}_{i} \in \{ A_{f_1},A_{f_2},\dots,A_{f_n} \}$, where $i \in \{1,\dots,k\}$ (for some positive integer  $k$) we get $R^{q}_{k}( \mathbb{Y}_{1},\dots ,\mathbb{Y}_{k} )=0$.
\end{defi}

\begin{defi} A family of self-adjoint operators $G_{f} =G_{f}^*$; $f \in \mathcal{H}$ is called $q$-Gaussian random variables if there exists a state $\mathbb{E}$ on the von Neumann algebra $\mathcal{A}$ (generated by $G_{f}$; $f \in \mathcal{H}$) such that the following  Wick formula holds:
 \begin{align}
\mathbb{E}(G_{f_1} \dots G_{f_n})=
\left\{ \begin{array}{ll} \sum_{\sigma \in P_2(n)}q^{rc(\sigma)}\prod_{j=1}^{n/2} \langle f_{j},f_{\sigma(j)} \rangle  & \textrm{if $n$ is even}
\\ 0 & \textrm{if $n$ is odd.}\\
\end{array} \right. \label{eq:qWickFormula}
\end{align}
where $P_2(n)$ is the set of 2-partitions of the set $\{1,2,\ldots,n\}$.
\end{defi}
The existence of such  random variables, far from being trivial, is ensured by Bo\.zejko and Speicher \cite{BS1}. Our assumptions on $\mathbb{E}$ do not allow us to use conditional expectations. In general, state $\mathbb{E}$ is not tracial so we do not know if conditional expectations exist. 
The following theorem is $q$-version of Theorem \ref{twr:1}.

%

\begin{theo}[]
Let $G_f$ and $G_g$ be two  $q$-independent variables with  
 $\mathbb{E}(G_f)=\mathbb{E}(G_g)=0$, $\mathbb{E}(G^2_f)=||f||^2=1$, $\langle f,g \rangle=0$, $\mathbb{E}(G^2_g)=||g||^2=1$ and $ R^{q}_{k}(G_g)=R^{q}_{k}(G_f)$ (for all integers $k\geqslant 0$ which means $G_f$ and $G_g$  have the same distribution)  then 
\begin{align}
\mathbb{E}((G_f-G_g)(G_f+G_g)(G_f-G_g)(G_f+G_g)^n)  =2q\mathbb{E}((G_f+G_g)^{n+1}) \textrm{ for all $n\geqslant 0$ }\label{eq:RownaniePomRodzial4}
\end{align} 
\label{twr:TwierdzenieDlaqasian}
if and only if  $G_f$ and $G_g$ are $q$-Gaussian random variables.
 \end{theo}
\begin{Rem}
If $G_f$ and $G_g$ are $q$-Gaussian random variables, then the state $\mathbb{E}$ is a trace and the conditional expectation exists  (see \cite{Br}). In this case  formula \eqref{eq:RownaniePomRodzial4} can be reformulated to
\begin{align}
\mathbb{E}((G_f-G_g)(G_f+G_g)(G_f-G_g)|(G_f+G_g))  =2q(G_f+G_g). 
\end{align} 
In particular, for $q=0$ we have the thesis of Theorem \ref{twr:1} is satisfied  for the variables with the same  distribution (only one way for $a=b=0$).
 \end{Rem}

\begin{Rem} The  Theorem \ref{twr:TwierdzenieDlaqasian} can be reformulated in the following form. 

Let $G_f$ and $G_g$ be two  $q$-independent variables and  
 $\mathbb{E}(G_f)=\mathbb{E}(G_g)=0$, $\mathbb{E}(G^2_f)=||f||^2$, $\langle f,g \rangle=0$, $\mathbb{E}(G^2_g)=||g||^2$ and $||f||^2 R^q_{k}(G_g)=||g||^2 R^q_{k}(G_f)$ (for all integer $k\geqslant 0$)  then 
\begin{align}
\mathbb{E}((||g||^2G_f-||f||^2G_g)(G_f+G_g)(||g||^2G_f-||f||^2G_g)(G_f+G_g)^n) \nonumber \\ =||g||^2+||f||^2)||f||^2||g||^2q\mathbb{E}((G_f+G_g)^{n+1}) \textrm{ for all $n\geqslant 0$ }
\end{align} 
if and only if  $G_f$ and $G_g$ are $q$-Gaussian random variables. The proof of this theorem is completely analogous with the one below, but is not as transparent as this expression presented beneath.
 \end{Rem}
\noindent  \textit{Proof of Theorem \ref{twr:TwierdzenieDlaqasian}.} From assumption of the Theorem \ref{twr:TwierdzenieDlaqasian}  above we deduce that $R^q_2(G_f)=||f||^2=1$, $R^q_2(G_g)=||g||^2=1$ and $R^q_2(G_f+G_g)=\langle f+g,f+g \rangle=2$.

 $\Leftarrow$:  Under the assumption that $G_f$ and $G_g$ are $q$-Gaussian random variables and let's compute 
\begin{align}
\mathbb{E}((G_f-G_g)(G_f+G_g)(G_f-G_g)(G_f+G_g)^n) \label{eq:RownaniePomRodzial5}.
\end{align}
From \eqref{eq:qWickFormula} we see that the both sides (left and right) of the main formula of Theorem \ref{twr:TwierdzenieDlaqasian}  are zero if $n$ is even thus we investigate only the case when $n$ is odd.  Since $R^q_2(G_f-G_g,G_f+G_g)=0$ (from the assumption $ R^q_{k}(G_g)= R^q_{k}(G_f)$) equation (\ref{eq:RownaniePomRodzial5}) equals $0$ if we consider 2-partitions  $\pi \in P_2(n+3)$ and $\pi= \{V_1, \dots,V_s\}$ where $V_1=\{1,k\}$ and $k\in \{2,4,5,\dots,n+3\}$ (if the first element is in  the 2-partition without third element). 
So we should analyse only this 2-partitions  $\pi \in P_2(n+3)$ such as  $\pi= \{V_1, \dots,V_s\}$ and $V_1=\{1,3\}$, see Figure \ref{fig:FiguraExemple3}. We denote this partitions by  $ P_2^{1,3}(n+3)$. Moreover this partitions  can be identified with the product $P_2(n+1)$ and $ P_2(2)$  multiplied by $q$ because if $1$ and $3$ are in the same block  we always  get one more crossing. 

\vspace{1cm}
\begin{figure}[h]
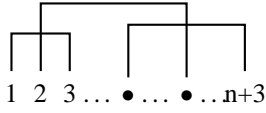

\begin{center}
\vspace{1cm}
\psset{nodesep=5pt}
\psset{angle=90}
\hspace{20mm}
\rnode{A}{\Node{1}} \rnode{A1}{\Node{2}}  \rnode{A2}{\Node{3}} \rnode{A6}{\Node{\dots}} \rnode{A7}{\Node{$\bullet$}} \rnode{A4}{\Node{\dots}} \rnode{A5}{\Node{$\bullet$}} \rnode{A8}{\Node{\dots}} \rnode{A9}{\Node{n+3}} 
\ncbar[armA=.5]{A}{A2}
\ncbar[armA=.9]{A1}{A5}
\ncbar[armA=.7]{A7}{A9}
\caption{The structure of crossing 2-partitions of $\{1, 2,3 ,\dots,n+3\}$ with $1$ and $3$ in the same block.}
\label{fig:FiguraExemple3}
\end{center}
\end{figure}

\noindent So if we use  \eqref {eq:qWickFormula} we get
\begin{flalign}
&\mathbb{E}((G_f-G_g)(G_f+G_g)(G_f-G_g)(G_f+G_g)^n)  \nonumber\\ &=
 \sum_{\sigma \in P_2^{1,3}(n+3)  }q^{rc(\sigma)}\langle f -g,f-g \rangle \langle f +g,f+g \rangle^{(n+1)/2}
\nonumber \\ &=
 2 q \sum_{\sigma \in P_2(n+1) }q^{rc(\sigma)} \langle f +g,f+g \rangle^{(n+1)/2}
 =2q\mathbb{E}((G_f+G_g)^{n+1}).
\end{flalign}
$\Rightarrow$:Suppose that  \eqref{eq:RownaniePomRodzial4} is true. 
Our proof relies on the observation that
\begin{align} R^q_{n+3}(G_f+G_g)=0 \label{eq:ZeroweKumulantyPomRodzial4}\end{align}  for all $n\geqslant 0$. We will prove
this by induction on the length of the cumulant.
Using the definition \eqref{eq:qCumulanty} and the assumption of $q$-independence and putting $n=0$ in \eqref{eq:RownaniePomRodzial4} we get 

\begin{align}
&\mathbb{E}((G_f-G_g)(G_f+G_g)(G_f-G_g))=R^q_{3}((G_f-G_g),(G_f+G_g),(G_f-G_g))\nonumber\\&=R^q_{3}((G_f+G_g),(G_f+G_g),(G_f+G_g))= 0. 
\end{align} 

We fix $k$ and suppose that \eqref{eq:ZeroweKumulantyPomRodzial4}  holds for all $n\in \{0,\dots,k\}$. Now we will prove equation \eqref{eq:ZeroweKumulantyPomRodzial4} for $n=k+1$. Expanding both sides of \eqref{eq:RownaniePomRodzial4}  into $q$-cumulants and using the fact that non-zero are only cumulants of size $2$ and $k+4$ we get  
\begin{align}
& \mathbb{E}((G_f-G_g)(G_f+G_g)(G_f-G_g)(G_f+G_g)^{k+1})\nonumber \\ &= \sum_{\sigma \in P_2(k+4)}q^{rc(\sigma)}R^q_2(G_f-G_g)\prod_{j=1}^{(k+2)/2} R^q_2(G_f+G_g)+R^q_{k+4}(G_f+G_g) \nonumber \\ &= 2q\sum_{\sigma \in P_2(k+2)}q^{rc(\sigma)}\prod_{j=1}^{(k+2)/2} R^q_2(G_f+G_g)+R^q_{k+4}(G_f+G_g) \nonumber \\  &\stackrel{\textrm{ right side of } \eqref{eq:RownaniePomRodzial4} }{=} 2q\mathbb{E}((G_f+G_g)^{k+2}) = 2q\sum_{\sigma \in P_2(k+2)}q^{rc(\sigma)}\prod_{j=1}^{(k+2)/2} R^q_2(G_f+G_g), 
\end{align} 
which  implies $R^q_{k+4}(G_f+G_g)=0$. Thus non-zero cumulants are only cumulants of size $2$ so we obtain that $G_f+G_g$ is a $q$-Gaussian random variable. From the assumption $ R^q_{k}(G_g)=R^q_{k}(G_f)$ we infer that $G_f$ and $G_g$ are $q$-Gaussian random variables as well. This completes the proof.
\begin{flushright} $\square$ \end{flushright}
%
\noindent \textbf{Open problems and remarks}
\begin{itemize}
\item In  Theorem \ref{twr:1} and Proposition  \ref{twr:4} of this paper we assume that the random variables are bounded that is $X_t \in \mathcal{A}$. It would be interesting  to show if this assumption can be replaced by $X_t \in L^2(\mathcal{A})$.
\item A version of Theorem \ref{twr:TwierdzenieDlaqasian} can be formulated for $q$-Poisson variables (see \cite{An2,An3}). The proof of this theorem is analogous with the proof of Theorem \ref{twr:TwierdzenieDlaqasian} (by induction). 
\item It would be worth to show whether  Proposition  \ref{twr:4} is true for non-commutative generalized stochastic processes with freely independent values, see \cite{BolyI,BolyII}.
\end{itemize}

\begin{center} Acknowledgments
\end{center}

The author would like to thank  M. Bo\.zejko, W. Bryc, M. Anshelevich, Z.Michna, W.M\l otkowski and J. Wysocza\'nski  for several discussions and helpful comments during the preparation of this paper. The author also thank especially the anonymous referee for very careful reading of the submitted manuscript.

\end{document}